\documentclass{scrartcl}

\usepackage{mathptmx}       
\usepackage{helvet}         
\usepackage{courier}        
\usepackage{type1cm}        
\usepackage{makeidx}         
\usepackage{graphicx}        
\usepackage{multicol}        
\usepackage[bottom]{footmisc}
\usepackage{tikz}

\makeindex             

\usepackage{amssymb, amsmath, amstext, amsopn, amsthm, amscd, amsxtra, amsfonts}
\usepackage{hyperref}%
\hypersetup{
urlcolor = red,
colorlinks = true,
linkcolor = blue,
citecolor = blue,
linktocpage = true,
pdftitle = {Overview of (pro-)Lie group structures on Hopf algebra character groups},
pdfauthor = {Geir Bogfjellmo, Rafael Dahmen, Alexander Schmeding},
bookmarksopen = true,
bookmarksopenlevel = 1,
unicode = true,
hypertexnames =false
}

\swapnumbers
\newtheoremstyle{standard}
{16pt} 
{16pt} 
{} 
{} 
{\bfseries}
{} 
{ } 
{{\thmname{#1~}}{\thmnumber{#2.}}\thmnote{~(#3)}} 

\newtheoremstyle{kursiv}
{16pt} 
{16pt} 
{\itshape} 
{} 
{\bfseries}
{} 
{ } 
{{\thmname{#1~}}{\thmnumber{#2.}}\thmnote{~(#3)}} 

\theoremstyle{standard}
\newtheorem{definition}[subsection]{Definition}
\newtheorem{example}[subsection]{Example}
\newtheorem{remark}[subsection]{Remark}
\newtheorem{nota}[subsection]{Notation}

\newtheorem{problem}[subsection]{Problem}

\theoremstyle{kursiv}
\newtheorem{theorem}[subsection]{Theorem}
\newtheorem{proposition} [subsection]{Proposition}

\newtheorem{lemma} [subsection]{Lemma}

\newcommand{\N}{\mathbb{N}}
\newcommand{\R}{\mathbb{R}}
\newcommand{\K}{\mathbb{K}}
\renewcommand{\C}{\mathbb{C}}

\newcommand{\cB}{\ensuremath{\mathcal{B}}}

\newcommand{\cG}{\ensuremath{\mathcal{G}}}
\newcommand{\cH}{\ensuremath{\mathcal{H}}}
\newcommand{\cI}{\ensuremath{\mathcal{I}}}
\newcommand{\cJ}{\ensuremath{\mathcal{J}}}
\newcommand{\cL}{\ensuremath{\mathcal{L}}}
\newcommand{\cP}{\ensuremath{\mathcal{P}}}
\newcommand{\cT}{\ensuremath{\mathcal{T}}}
\newcommand{\cU}{\ensuremath{\mathcal{U}}}

\newcommand{\Hopf}  {\cH}					
\newcommand{\HIdeal}{\cJ}			                

\newcommand{\lcA}{A}						
\newcommand{\lcB}{B}						

\newcommand{\g}{\mathfrak{g}}
\DeclareMathOperator{\Hom}{Hom}
\DeclareMathOperator*{\Ann}{Ann}

\newcommand{\InfChar}[2]{ \g(#1 , #2) }				
\newcommand{\Char}[2]{     \mathcal{G}(#1 , #2) }			

\newcommand{\AnnGroup}{\Ann(\HIdeal, \lcB) \cap \Char{\Hopf}{\lcB}}
\newcommand{\AnnLieAlgebra}{\Ann(\HIdeal, \lcB) \cap \InfChar{\Hopf}{\lcB}}

\newcommand{\LB}[1][\cdot \hspace{1pt} , \cdot]{[\hspace{1pt} #1 \hspace{1pt} ]}

\newcommand{\Evol}{\mathrm{Evol}}
\newcommand{\evol}{\mathrm{evol}}
\newcommand{\ev}{\mathrm{ev}}
\newcommand{\Lf}{\ensuremath{\mathbf{L}}}
\DeclareMathOperator{\one}{\mathbf{1}}

\DeclareMathOperator*{\Mero}{\mathcal{K}}
\DeclareMathOperator*{\Holo}{\mathcal{O}}
\newcommand{\BHol}{\Hol_{\mathrm{b}}}
\newcommand{\Hol}{\mathrm{Hol}}

\newcommand{\norm}[1]{\left\| #1 \right\|}
\newcommand*\D{\mathop{}\!\mathrm{d}}
\newcommand{\RT}{\ensuremath{\cT}}
\DeclareMathOperator{\OST}{OST}

\newcommand{\BGp}{\ensuremath{G_{\mathrm{TM}}}}

\tikzstyle dtree=[grow'=up,sibling distance=2mm,level distance=2mm,thick]
\tikzstyle dtree node=[scale=0.3,shape=circle,very thin,draw]
\tikzstyle dtree black node=[style=dtree node,fill=black]
\newcommand{\onenode}{
  \begin{tikzpicture}[dtree]
    \node[dtree black node] {}
    ;
  \end{tikzpicture}
}

\newcommand{\twonode}{
\begin{tikzpicture}[dtree]
  \node[dtree black node] {}
  child { node[dtree black node] {} }
  ;
\end{tikzpicture}
}

\makeatletter
\providecommand*{\shuffle}{%
  \mathbin{\mathpalette\shuffle@{}}%
}
\newcommand*{\shuffle@}[2]{%
  \sbox0{$#1\vcenter{}$}%
  \kern .15\ht0 
  \rlap{\vrule height .25\ht0 depth 0pt width 2.5\ht0}%
  \raise.1\ht0\hbox to 2.5\ht0{%
    \vrule height 1.75\ht0 depth -.1\ht0 width .17\ht0 %
    \hfill
    \vrule height 1.75\ht0 depth -.1\ht0 width .17\ht0 %
    \hfill
    \vrule height 1.75\ht0 depth -.1\ht0 width .17\ht0 %
  }%
  \kern .15\ht0 
}
\makeatother

\begin{document}

\title{Overview of (pro-)Lie group structures on Hopf algebra character groups}
\author{Geir Bogfjellmo\footnote{\quad Chalmers Technical University \& Gothenburg University, Chalmers Tv\"a rgata 3, Gothenburg, Sweden,\newline current affiliation: ICMAT Madrid, Spain \href{mailto:geir.bogfjellmo@icmat.es}{geir.bogfjellmo@icmat.es}}, Rafael Dahmen\footnote{\quad TU Darmstadt, Schlo\ss{}gartenstr.\ 7, Darmstadt, Germany, \href{mailto:dahmen@mathematik.tu-darmstadt.de}{dahmen@mathematik.tu-darmstadt.de}}, Alexander Schmeding\footnote{\quad NTNU, Alfred Getz vei 1, Trondheim, Norway,\newline current affiliation: TU Berlin, Germany \href{mailto:schmeding@tu-berlin.de}{schmeding@tu-berlin.de}}}

\maketitle

\abstract{
Character groups of Hopf algebras appear in a variety of mathematical and physical contexts.
To name just a few, they arise in non-commutative geometry, renormalisation of quantum field theory, and numerical analysis.
In the present article we review recent results on the structure of character groups of Hopf algebras as infinite-dimensional (pro-)Lie groups.
It turns out that under mild assumptions on the Hopf algebra or the target algebra the character groups possess strong structural properties.
Moreover, these properties are of interest in applications of these groups outside of Lie theory.
We emphasise this point in the context of two main examples:
  \begin{itemize}
   \item the Butcher group from numerical analysis and
   \item character groups which arise from the Connes--Kreimer theory of renormalisation of quantum field theories.
  \end{itemize}
}

\tableofcontents

\section*{Introduction}
\label{sec:1}
Character groups of Hopf algebras appear in a variety of mathematical and physical contexts.
 To name just a few, they arise in non-commutative geometry, renormalisation of quantum field theory (see \cite{MR2371808}) and numerical analysis (cf.\ \cite{Brouder-04-BIT}).
  In these contexts the arising groups are often studied via an associated Lie algebra or by assuming an auxiliary topological or differentiable structure on these groups.
 For example, in the context of the Connes--Kreimer theory of renormalisation of quantum field theories, the group of characters of a Hopf algebra of Feynman graphs is studied via its Lie algebra (cf.\ \cite{MR2371808}).
 Moreover, it turns out that this group is a projective limit of finite-dimensional Lie groups.
 The (infinite-dimensional) Lie algebra and the projective limit structure are important tools to analyse these character groups.

 Another example for a character group is the so called Butcher group from numerical analysis which can be realised as a character group of a Hopf algebra of trees.
 In \cite{BS14} we have shown that the Butcher group can be turned into an infinite-dimensional Lie group.
 All of these results can be interpreted in the general framework of (infinite-dimensional) Lie group structures on character groups of Hopf algebras developed in \cite{BDS15}.
 \medskip

 The present article provides an introduction to the theory developed in \cite{BDS15} with a view towards applications and further research.
 To be as accessible as possible, we consider the results from the perspective of examples arising in numerical analysis and mathematical physics.
 Note that the Lie group structures discussed here will in general be infinite-dimensional and their modelling spaces will be more general than Banach spaces.
 Thus we base our investigation on a concept of $C^r$-maps between locally convex spaces known as Bastiani calculus or Keller's $C^r_c$-theory.
 However, we recall and explain the necessary notions as we do not presuppose familiarity with the concepts of infinite-dimensional analysis and Lie theory.

 We will always suppose that the target algebra supports a suitable topological structure, i.e.\ it is a locally convex algebra (e.g.\ a Banach algebra).
  The infinite-dimensional structure of a character group is then determined by the algebraic structure of the source Hopf algebra and the topological structure of the target algebra.
 If the Hopf algebra is graded and connected, it turns out that the character group (with values in a locally convex algebra) can be made into an infinite-dimensional Lie group.
 Hopf algebras with these properties and their character groups appear often in combinatorial contexts\footnote{In these contexts the term ``combinatorial Hopf algebra'' is often used though there seems to be no broad consensus on the meaning of this term.} and are connected to certain operads and applications in numerical analysis.
 For example in \cite{MR3350091} Hopf algebras related to numerical integration methods and their connection to pre-Lie and post-Lie algebras are studied.

 If the Hopf algebra is not graded and connected, then the character group can in general not be turned into an infinite-dimensional Lie group.
 However, it turns out that the character group of an arbitrary Hopf algebra (with values in a finite-dimensional algebra) is always a topological group with strong structural properties, i.e.\ it is always the projective limit of finite-dimensional Lie groups.
 Groups with these properties -- so called pro-Lie groups -- are accessible to Lie theoretic methods (cf.\ \cite{MR2337107}) albeit they may not admit a differential structure.

 Finally, let us summarise the broad perspective taken in the present article by the slogan:
 Many constructions on character groups can be interpreted in the framework of infinite-dimensional (pro-)Lie group structures on these groups.

\section{Hopf algebras and characters}

In this section, we recall the language of Hopf algebras and their character groups.
The material here is covered by almost every introduction to Hopf algebras (see \cite{MR547117,MR2523455,MR2290769,MR1321145,MR1381692,MR0252485}).
Thus we restrict the exposition here to fix only some notation.

\begin{nota}
Denote by $\N$ the set of natural numbers (without $0$) and $\N_0 := \N \cup \{0\}$.
By $\K$ we denote either the field of real numbers $\R$ or of complex numbers $\C$.
\end{nota}

\begin{definition}[Hopf algebra]	
 A Hopf algebra $\Hopf=(\Hopf,m_\Hopf,u_\Hopf,\Delta_\Hopf,\epsilon_\Hopf,S_\Hopf)$ (over $\K$) is a bialgebra with compatible antipode $S \colon \Hopf \rightarrow \Hopf$,
 i.e.\ $(\Hopf, m_\Hopf, u_\Hopf)$ is a unital $\K$-algebra, $(\Hopf, \Delta_\Hopf, \epsilon_\Hopf)$ is a unital $\K$-coalgebra and the following holds:
  \begin{enumerate}
   \item the maps $\Delta_\Hopf \colon \Hopf \rightarrow \Hopf \otimes \Hopf$ and $\epsilon_\Hopf \colon \Hopf \rightarrow \K$ are algebra morphisms
   \item $S$ is $\K$-linear with $m_\Hopf \circ (\mbox{id} \otimes S) \circ  \Delta_\Hopf  = u_\Hopf \circ \epsilon_\Hopf = m_\Hopf \circ (S \otimes \mbox{id} ) \circ  \Delta_\Hopf$.
  \end{enumerate}

 A Hopf algebra $\Hopf$ is ($\N_0$-)\emph{graded}, if it admits a decomposition $\Hopf = \oplus_{n \in \N_0} \Hopf_n$ such that $m_\Hopf (\Hopf_n \otimes \Hopf_m) \subseteq \Hopf_{n+m}$ and $\Delta_\Hopf (\Hopf_n) \subseteq \bigoplus_{k+l=n} \Hopf_l \otimes \Hopf_k$ hold for all $n,m \in \N_0$.
 If in addition $\Hopf_0 \cong \K$ is satisfied $\Hopf$ is called \emph{graded and connected} Hopf algebra.
\end{definition}

In the Connes--Kreimer theory of renormalisation of quantum field theory one considers the Hopf algebra $\Hopf_{FG}$ of Feynman graphs.\footnote{This Hopf algebra depends on the quantum field theory under consideration, but we will suppress this dependence in our notation.} As it is quite involved to define this Hopf algebra we refer to \cite[1.6]{MR2371808} for details. Below we describe in Example \ref{ex: RT} a related but simpler Hopf algebra.
This Hopf algebra of rooted trees arises naturally in numerical analysis, renormalisation of quantum field theories and non-commutative geometry (see \cite{Brouder-04-BIT} for a survey).
We recall the definition of the Hopf algebra in broad strokes as it is somewhat prototypical for Hopf algebras from numerical analysis.
To construct the Hopf algebra, recall some notation first.

 \begin{nota}
 \begin{enumerate}
  \item A \emph{rooted tree} is a connected \emph{finite} graph without cycles with a distinguished node called the \emph{root}.
  We identify rooted trees if they are graph isomorphic via a root preserving isomorphism.

  Let $\RT$ be \emph{the set of all rooted trees} and write $\RT_0 := \RT \cup \{\emptyset\}$ where $\emptyset$ denotes the empty tree.
  The \emph{order} $|\tau|$ of a tree $\tau \in \RT_0$ is its number of vertices.
  \item An \emph{ordered subtree}\footnote{The term ``ordered'' refers to that the subtree remembers from which part of the tree it was cut.} of $\tau \in \RT_0$ is a subset $s$ of all vertices of $\tau$ which satisfies
    \begin{itemize}
     \item[(i)] \ $s$ is connected by edges of the tree $\tau$,
     \item[(ii)] \ if $s$ is non-empty, then it contains the root of $\tau$.
    \end{itemize}
   The set of all ordered subtrees of $\tau$ is denoted by $\OST (\tau)$.
   Further, $s_\tau$ denotes the tree given by vertices of $s$ with root and edges induced by $\tau$.
  \item A \emph{partition} $p$ of a tree $\tau \in \RT_0$ is a subset of edges of the tree.
 We denote by $\cP (\tau)$ the set of all partitions of $\tau$ (including the empty partition).
 \end{enumerate}
  Associated to $s \in \OST (\tau)$ is a forest $\tau \setminus s$ (collection of rooted trees) obtained from $\tau$ by removing the subtree $s$ and its adjacent edges.
  Similarly, to a partition $p \in \cP (\tau)$ a forest $\tau \setminus p$ is associated as the forest that remains when the edges of $p$ are removed from the tree $\tau$.
  In either case, we let $\# \tau \setminus p$ be the number of trees in the forest.
\end{nota}

 \begin{example}[The Connes--Kreimer Hopf algebra of rooted trees \cite{CK98}]\label{ex: RT}
  Consider the algebra $\Hopf^{\K}_{CK} := \K [\cT]$ of polynomials which is generated by the trees in $\RT$.
  We denote the structure maps of this algebra by $m$ (multiplication) and $u$ (unit).
  Indeed $\Hopf^\K_{CK}$ becomes a bialgebra whose coproduct we define on the trees (which generate the algebra) via
    \begin{displaymath}
     \Delta \colon \Hopf^\K_{CK} \rightarrow \Hopf^\K_{CK} \otimes \Hopf^\K_{CK} ,\quad  \tau \mapsto \sum_{s \in \OST (\tau)} (\tau \setminus s) \otimes s_\tau .
    \end{displaymath}
 Then the counit $\epsilon \colon \Hopf^\K_{CK} \rightarrow \K$ is given by $\epsilon (1_{\Hopf^\K_{CK}}) = 1$ and $\epsilon (\tau) = 0$ for all $\tau \in \RT$.
 Finally, the antipode is defined on the trees (which generate $\Hopf^\K_{CK}$) via
  \begin{displaymath}
   S \colon \Hopf^\K_{CK} \rightarrow \Hopf_{CK}^\K ,\quad \tau \mapsto \sum_{p \in \cP (\tau)} (-1)^{\# \tau \setminus p} (\tau \setminus p)
  \end{displaymath}
 and one can show that $\Hopf^\K_{CK} = (\Hopf^\K_{CK} ,m,u, \Delta,\epsilon,S)$ is a $\K$-Hopf algebra (see \cite[5.1]{CHV2010} for more details and references).

 Furthermore, $\Hopf^\K_{CK}$ is a graded and connected Hopf algebra with respect to the \emph{number of nodes grading}:
 For each $n \in \N_0$ we define the $n$th degree via
  \begin{displaymath}
  \mbox{for } \tau_i \in \RT, 1 \leq i \leq k, k \in \N_0 \quad \quad \tau_1 \cdot \tau_2 \cdot \ldots \cdot \tau_k  \in (\Hopf^\K_{CK})_n \mbox{ if and only if } \sum_{r= 1}^k |\tau_k| = n
  \end{displaymath}
 \end{example}

\begin{remark}
 Recently modifications of the Hopf algebra of rooted trees have been studied in the context of numerical analysis.
 In particular, non-commutative analogues to the Hopf algebra $\Hopf_{CK}^\K$ (where the algebra is constructed as the non-commutative polynomial algebra of planar trees) have been studied in the context of Lie--Butcher theory.
 Their groups of characters, which we define in a moment, are of particular interest to develop techniques for numerical integration on manifolds (see \cite{MR2407032}).
\end{remark}

 We briefly mention another example of a ``combinatorial'' Hopf algebra:

 \begin{example}[{The shuffle Hopf algebra \cite{MR1231799}}]\label{ex: shuffle}
  Fix a non-empty set $A$ called the \emph{alphabet}.
  A \emph{word} in the alphabet $A$ is a finite (possibly empty) sequence of elements in $A$.
  We denote by $A^*$ the set of all words in $A$ and consider the vector space $\C \langle A\rangle$ freely generated by the elements in $A^*$.
  Let $w = a_1 \ldots a_n$ and $w'  = a_{n+1} \ldots a_{n+m}$ be words of length $n$ and $m$, respectively.
  Define the \emph{shuffle product} by
    \begin{displaymath}
     w \shuffle w' := \sum_{\{i_1 < i_2 < \cdots < i_n\}, \{j_1 < j_2 < \cdots < j_m\}} a_{i_1} \cdots a_{i_n} a_{j_1} \cdots a_{j_m} ,
    \end{displaymath}
  where the summation runs through all disjoint sets which satisfy
    \begin{displaymath}
     \{i_1 < i_2 < \cdots < i_n\} \sqcup \{j_1 < j_2 < \cdots < j_m\} = \{1,\ldots, m+n\}.
    \end{displaymath}
  The (bilinear extensions of the) shuffle product and the deconcatenation of words turns $\C \langle A\rangle$ into a complex bialgebra.
  Together with the antipode $S (w) = (-1)^n w$ (for a word $w= a_n \ldots a_1$ of length $n$) and the grading by word length, we obtain a graded and connected Hopf algebra $\text{Sh} (A)$.
  We call this Hopf algebra the \emph{shuffle Hopf algebra}.

  Recall that the shuffle Hopf algebra appears in diverse applications connected to numerical analysis, see e.g.\ \cite{MR2407032,1502.05528v2,MR2790315}.
 \end{example}

 We will now consider groups of characters of Hopf algebras:
 \begin{definition}							\label{def: character}
  Let $\Hopf$ be a Hopf algebra and $\lcB$ a commutative algebra.
  A linear map $\phi \colon \Hopf \rightarrow \lcB$ is called
  	\begin{enumerate}
  	\item ($\lcB$-valued) \emph{character} if it is a homomorphism of unital algebras, i.e.
  \begin{equation}\label{eq char:char}
   \phi(a b) = \phi(a)\phi(b)\mbox{ for all } a,b \in \Hopf \mbox{ and } \phi(1_\Hopf)=1_\lcB.
  \end{equation}
  The set of characters is denoted by $\Char{\Hopf}{\lcB}$.
  \item \emph{infinitesimal character} if
  \begin{equation}\label{eq: InfChar:char}
  \phi(a b)=\phi(a) \epsilon_\Hopf(b) + \epsilon_\Hopf(a) \phi(b) \mbox{ for all }  a,b \in \Hopf
  \end{equation}
  We denote by  $\InfChar{\Hopf}{\lcB}$ the set of all infinitesimal characters.
  \end{enumerate}
 \end{definition}

 \begin{lemma}[{\cite[4.3 Proposition 21 and 22]{MR2523455}}]
 Let $\Hopf$ be a Hopf algebra and $\lcB$ a commutative algebra with multiplication map $\mu_\lcB \colon \lcB \otimes \lcB \rightarrow \lcB$.
 Then we obtain the algebra of $\K$-linear maps $\Hom_\K(\Hopf,\lcB)$ with the \emph{convolution product}
 \[
  \phi \star \psi := \mu_{\lcB} \circ (\phi \otimes \psi) \circ \Delta_\Hopf
 \]
 and the following holds:
 \begin{enumerate}
 \item  $\Char{\Hopf}{\lcB}$ is a subgroup of the group of units $(\Hom_{\K} (\Hopf,\lcB)^\times,\star)$.
  Inversion in  $\Char{\Hopf}{\lcB}$ is given by $\phi \mapsto \phi \circ S_\Hopf$ and $1_\lcA:= u_\lcB\circ\epsilon_\Hopf \colon \Hopf \rightarrow \lcB , x \mapsto \epsilon_\Hopf (x) 1_\lcB$ is the unit element.
  \item  $\InfChar{\Hopf}{\lcB}$ is a Lie subalgebra of $(\Hom_\K (\Hopf,\lcB),[\cdot,\cdot])$, where $\LB{}$ is the commutator bracket of $(A,\star)$.
 \end{enumerate}
 \end{lemma}

 \begin{example}[Character groups of Hopf algebras]
   \begin{enumerate}
   \item The universal enveloping algebra $\cU (\g)$ of a Lie algebra $\g$ over $\K$ is a Hopf algebra (see \cite[3.6]{MR2290769})\footnote{Note that $\g$ can be recovered from the Hopf algebra $\cU (\g)$ (see \cite[Theorem 3.6.1]{MR2290769}) and the Hopf algebras which arise in this way are characterised by the Milnor--Moore theorem (cf.\ Remark \ref{rem: MMthm}). }.
   Every character $\phi\in\Char{\cU(\g)}{\K}$ corresponds to a Lie algebra homomorphism $\phi|_\g \colon \g \rightarrow \K$ which in turn factors naturally through a linear map\\ $\phi^\sim \colon \g/([\g,\g]) \rightarrow \K$, yielding a group isomorphism
   \[
    \Phi \colon \Char{\cU(\g)}{\K} \rightarrow \bigl((\g/[\g,\g])^*,+\bigr),\quad \phi \mapsto \bigl( \phi^\sim\colon v+[\g,\g] \mapsto \phi(v) \bigr).
   \]
   Thus we can identify the character group with the additive group of the dual vector space of the abelianisation of $\g$.
   (The group  $\Char{\cU(\g)}{\K}\cong \g/[\g,\g]$ is also naturally isomorphic to the first cohomology group $H^1(\g,\K)$ of the Lie algebra $\g$ with values in the trivial module $\K$, cf.\ \cite[p.168]{MR938524}.)
   \item Character groups of the Hopf algebra of Feynman graphs $\Hopf_{FG}$ play an important role in renormalisation of quantum field theories.
   Namely in the mathematical formulation of the renormalisation procedure one considers the groups $\Char{\Hopf_{FG}}{\C}$ and $\Char{\Hopf_{FG}}{\Mero}$, where $\Mero$ is the algebra of germs of meromorphic functions.
   We will return to these groups in Examples \ref{ex: mero} and \ref{ex: toe}.
   \item The character group $\Char{\Hopf_{CK}^\R}{\R}$ of the real Hopf algebra of rooted trees $\Hopf_{CK}^\R$ from Example \ref{ex: RT} turns out to be the Butcher group $\BGp$ from numerical analysis.
   This group is connected to numerical integration theory (see \cite{Brouder-04-BIT}).
   \item The character group of the shuffle algebra $\text{Sh} (A)$ (Example \ref{ex: shuffle}) appears for example in \cite{1502.05528v2}.
   There, the character group has been studied in the context of dynamical systems and their discretisation.
   We will provide more details below in Example \ref{ex: shuffle:char}.
   Moreover, in \cite{MR2790315} these character groups are considered in the context of Lie-Butcher theory.
  \end{enumerate}\label{ex: charactergroups}
 \end{example}

 Note that the concepts recalled in this section were purely algebraic and combinatorial in nature.
 In particular, we have neither referred to a topology nor to a differentiable structure on these groups.
 We will introduce the necessary tools (differential calculus on locally convex spaces) to endow the character groups of Hopf algebras with an infinite-dimensional Lie group structure in the next section.

\section{A primer on infinite-dimensional differential calculus and Lie groups}

 In this section basic facts on the differential calculus in infinite-dimensional spaces are recalled.
 The general setting for our calculus are locally convex spaces (see \cite{MR0342978, jarchow1980}).

 \begin{definition}
  Let $E$ be a vector space over $\K \in \{\R,\C \}$ together with a topology $T$.
  \begin{enumerate}
   \item $(E,T)$ is called \emph{topological vector space}, if the vector space operations are continuous with respect to $T$ and the metric topology on $\K$.
   \item A \emph{seminorm} on $E$ is a map $p \colon E \rightarrow [0,\infty[$ which satisfies for $x,y \in E, \lambda \in \K$ the identities $p(x+y) \leq p(x)+p(y)$ and $p(\lambda x)  =|\lambda| p(x)$, but may have $p^{-1} (0) \neq \{0\}$.
   We denote by $\text{pr}_p \colon E \rightarrow E_p := E/p^{-1} (0)$ the canonical projection to the normed space associated to $p$.
   \item A topological vector space $(E,T)$ is called \emph{locally convex space} if there is a family $\{p_i \mid i \in I\}$ of continuous seminorms for some index set $I$, such that
      \begin{enumerate}
      \item[i.] the topology $T$ is the initial topolgy with respect to $\{\text{pr}_{p_i} \colon E \rightarrow E_{p_i} \mid i \in I\}$, i.e.\ the $E$-valued (]map $f$ is continuous if and only if $\text{pr}_i \circ f$ is continuous for each $i\in I$,
      \item[ii.] if $x \in E$ with $p_i (x) = 0$ for all $i \in I$, then $x = 0$ (i.e.\ $T$ is Hausdorff).
     \end{enumerate}
  In this case, the topology $T$ is \emph{generated by the family of seminorms} $\{p_i\}_{i \in I}$.
  Usually we suppress $T$ and write $(E, \{p_i\}_{i\in I})$ or simply $E$ instead of $(E,T)$.
  \end{enumerate}
 \end{definition}

 Many familiar results from finite-dimensional calculus carry over to infinite dimensions if we assume that all spaces are locally convex.
 Hence we will only consider topological vector spaces which are locally convex.
 Note that the term \textquotedblleft locally convex\textquotedblright\ comes from the fact that the semi-norm balls
 form convex neighbourhoods of the points.

 \begin{example}
  \begin{enumerate}
   \item Normed spaces are locally convex spaces (see \cite[Ch.\ I 6.2]{MR0342978}).
   \item Let $(E, \{p_i\}_{i \in I})$ be a locally convex vector space and $X$ a set.
   Consider the space $E^X$ of all mappings from $X$ to $E$.
   For $x \in X$ we define the point-evaluation $\ev_x \colon E^X \rightarrow E, f \mapsto f (x)$.
  The \emph{topology of pointwise convergence} on $E^X$ is the locally convex topology generated by the seminorms
   	\begin{displaymath}
   	 q_{i,x} := p_i \circ \ev_x \mbox{ where } (i,x) \mbox{ runs through } I \times X.
   	 \end{displaymath}
   With the pointwise vector space operations and the topology of pointwise convergence $E^X$ becomes a locally convex vector space.
   By definition, this topology is completely determined by the topology of the target vector space.

   If $X$ is a linear space, we consider the subspace $\Hom_\K (X,E)$ of all linear  maps in $E^X$ with the induced topology (also called topology of pointwise convergence).
   \end{enumerate}\label{ex: lcvx:spac}
 \end{example}

  As we are working beyond the realm of Banach spaces, the usual notion of Fr\'{e}chet-differentiability cannot be used.\footnote{The problem here is that the bounded linear operators do not admit a good topological structure if the spaces are not normable. In  particular, the chain rule will not hold for Fr\'{e}chet-differentiability in general for these spaces (cf.\ \cite[p. 73]{MR583436} or \cite{keller1974}).}
  Moreover, there are several inequivalent notions of differentiability on locally convex spaces (see \cite{keller1974}).
  For more information on our setting of differential calculus we refer the reader to \cite{MR1911979,keller1974}.
  The notion of differentiability we adopt is natural and quite simple, as the derivative is defined via directional derivatives.

\begin{definition}\label{defn: deriv}
 Let $\K \in \{\R,\C\}$, $r \in \N \cup \{\infty\}$ and $E$, $F$ locally convex $\K$-vector spaces and $U \subseteq E$ open.
 Moreover we let $f \colon U \rightarrow F$ be a map.
 If it exists, we define for $(x,h) \in U \times E$ the directional derivative
 $$df(x,h) := D_h f(x) := \lim_{\K^\times \ni t\rightarrow 0} t^{-1} (f(x+th) -f(x)).$$
 We say that $f$ is $C^r_\K$ if the iterated directional derivatives
    \begin{displaymath}
     d^{(k)}f (x,y_1,\ldots , y_k) := (D_{y_k} D_{y_{k-1}} \cdots D_{y_1} f) (x)
    \end{displaymath}
 exist for all $k \in \N_0$ such that $k \leq r$, $x \in U$ and $y_1,\ldots , y_k \in E$ and define continuous maps $d^{(k)} f \colon U \times E^k \rightarrow F$.
 If it is clear which $\K$ is meant, we simply write $C^r$ for $C^r_\K$.
 If $f$ is $C^\infty_\C$, we say that $f$ is \emph{holomorphic} and if $f$ is $C^\infty_\R$ we say that $f$ is \emph{smooth}.
\end{definition}

On Fr\'{e}chet spaces (i.e.\ complete metrisable locally convex spaces) our notion of differentiability coincides with that from the \textquotedblleft convenient setting\textquotedblright\ of global analysis outlined in \cite{MR1471480}.
Note that differentiable maps in our setting are continuous by default (which is in general not true in the convenient setting).
Later on we will need analytic mappings between infinite-dimensional spaces. Note first:

\begin{remark}
 A map $f \colon U \rightarrow F$ is of class $C^\infty_\C$ if and only if it is \emph{complex analytic} i.e.,
  if $f$ is continuous and locally given by a series of continuous homogeneous polynomials (cf.\ \cite[Proposition 7.4 and 7.7]{MR2069671}).
  We then also say that $f$ is of class $C^\omega_\C$.
\label{rem: analytic}
\end{remark}

To introduce real analyticity, we have to generalise a suitable characterisation from the finite-dimensional case:
A map $\R \rightarrow \R$ is real analytic if it extends to a complex analytic map $\C \supseteq U \rightarrow \C$ on an open $\R$-neighbourhood $U$ in $\C$.
We proceed analogously for locally convex spaces by replacing $\C$ with a suitable complexification.

\begin{definition}[Complexification of a locally convex space]
 Let $E$ be a real locally convex topological vector space.
 Endow $E_\C := E \times E$ with the following operation
 \begin{displaymath}
  (x+iy).(u,v) := (xu-yv, xv+yu) \quad \mbox{ for } x,y \in \R, u,v \in E.
 \end{displaymath}
 The complex vector space $E_\C$ with the product topology is called the \emph{complexification} of $E$. We identify $E$ with the closed real subspace $E\times \{0\}$ of $E_\C$.
\end{definition}

\begin{definition}
 Let $E$, $F$ be real locally convex spaces and $f \colon U \rightarrow F$ defined on an open subset $U$.
 Following \cite{MR830252} and \cite{MR1911979}, we call $f$ \emph{real analytic} (or $C^\omega_\R$) if $f$ extends to a $C^\infty_\C$-map $\tilde{f}\colon \tilde{U} \rightarrow F_\C$ on an open neighbourhood $\tilde{U}$ of $U$ in the complexification $E_\C$.\footnote{If $E$ and $F$ are Fr\'{e}chet spaces, real analytic maps in the sense just defined coincide with maps which are continuous and can be locally expanded into a power series. See \cite[Proposition 4.1]{MR2402519}.}
\end{definition}

Note that many of the usual results of differential calculus carry over to our setting.
In particular, maps on connected domains whose derivative vanishes are constant as a version of the fundamental theorem of calculus holds.
Moreover, the chain rule holds in the following form:

\begin{lemma}[{Chain Rule {\cite[Propositions 1.12, 1.15, 2.7 and 2.9]{MR1911979}}}]
 Fix $k \in \N_0 \cup \{\infty , \omega\}$ and $\K \in \{\R , \C\}$ together with $C^k_\K$-maps $f \colon E \supseteq U \rightarrow F$ and $g \colon H \supseteq V \rightarrow E$ defined on open subsets of locally convex spaces.
 Assume that $g (V) \subseteq U$. Then $f\circ g$ is of class $C^{k}_\K$ and the first derivative of $f\circ g$ is given by
  \begin{displaymath}
   d(f\circ g) (x;v) = df(g(x);dg(x,v)) \quad  \mbox{for all } x \in U ,\ v \in H.
  \end{displaymath}
\end{lemma}

The differential calculus developed so far extends easily to maps which are defined on non-open sets.
This situation occurs frequently in the context of differential equations on closed intervals (see \cite{alas2012} for an overview).

Having the chain rule at our disposal we can define manifolds and related constructions which are modelled on locally convex spaces.

\begin{definition} Fix a Hausdorff topological space $M$ and a locally convex space $E$ over $\K \in \{\R,\C\}$.
An ($E$-)manifold chart $(U_\kappa, \kappa)$ on $M$ is an open set $U_\kappa \subseteq M$ together with a homeomorphism $\kappa \colon U_\kappa \rightarrow V_\kappa \subseteq E$ onto an open subset of $E$.
Two such charts are called $C^r$-compatible for $r \in \N_0 \cup \{\infty,\omega\}$ if the change of charts map $\nu^{-1} \circ \kappa \colon \kappa (U_\kappa \cap U_\nu) \rightarrow \nu (U_\kappa \cap U_\nu)$ is a $C^r$-diffeomorphism.
A $C^r$-atlas of $M$ is a set of pairwise $C^r$-compatible manifold charts, whose domains cover $M$. Two such $C^r$-atlases are equivalent if their union is again a $C^r$-atlas.

A \emph{locally convex $C^r$-manifold} $M$ modelled on $E$ is a Hausdorff space $M$ with an equivalence class of $C^r$-atlases of ($E$-)manifold charts.
\end{definition}

 Direct products of locally convex manifolds, tangent spaces and tangent bundles as well as $C^r$-maps of manifolds may be defined as in the finite dimensional setting (see \cite[I.3]{MR2261066}).
 The advantage of this construction is that we can now give a very simple answer to the question, what an infinite-dimensional Lie group is:

\begin{definition}
A (locally convex) \emph{Lie group} is a group $G$ equipped with a $C_\K^\infty$-manifold structure modelled on a locally convex space, such that the group operations are smooth.
If the manifold structure and the group operations are in addition ($\K$-) analytic, then $G$ is called a ($\K$-) \emph{analytic} Lie group.
\end{definition}

Later on, we will encounter special classes of infinite-dimensional Lie groups.
We recommend \cite{MR2261066} for a survey on the theory of locally convex Lie groups.

\section{The Lie group structure of character groups of Hopf algebras}

 In this section we recall and explain how to construct a Lie group structure on character groups of Hopf algebras (cf.\ \cite{BDS15}).
 Let us first construct a suitable topology on the character groups which will in the end turn the character group into a Lie group.
 As our notion of differentiability is built on top of continuity (i.e.\ every differentiable map is continuous), the topology needs to turn the character group into a topological group.

 \begin{remark}
 Consider the character group $\Char{\Hopf}{\lcB}$ of a Hopf algebra $\Hopf$ with values in the commutative algebra $\lcB$.
 If $\lcB$  is a locally convex vector space, we  endow the space $\Hom_\K (\Hopf , \lcB)$ with the topology of pointwise convergence (see Example \ref{ex: lcvx:spac}).
 This topology turns $\InfChar{\Hopf}{\lcB}$ and $\Char{\Hopf}{\lcB}$ into closed subsets of $\Hom_\K (\Hopf, \lcB)$ which will always be endowed with the topology of pointwise convergence.
 \end{remark}

 To see that this topology turns the character groups into topological groups we need the group product, the convolution $\star$, to be continuous.
 The product $\star$ is defined via the multiplication of the algebra $\lcB$.
 Thus we will see that $\star$ is continuous if the locally convex structure on $\lcB$ is compatible with the algebra structure, i.e.\ the algebra $\lcB$ is a locally convex algebra.

 \begin{definition}								\label{defn_locally_convex_algebra}
  Let $\lcB$ be an algebra over $\K \in \{ \R,\C\}$. We call $\lcB$ \emph{locally convex algebra} if $\lcB$ is a locally convex vector space such that the algebra product $\mu_\lcB \colon \lcB \times \lcB \rightarrow \lcB$ is a continuous bilinear map.
 \end{definition}

 \begin{remark}
  We have not asked the algebra product to be a continuous linear map on $\lcB \otimes \lcB$ since this would require us to choose a topology on the tensor product.
  As there are different valid choices for topological tensor products between locally convex spaces we refrain from doing this.
  In fact, the whole point of the definition is to avoid topological tensor products.\footnote{Note that it is well known that the algebra multiplication as a bilinear map will be continuous if we require it to be continuous with respect to the projective topological tensor product $B \otimes_\pi B$ (see e.g.\ \cite{MR2296978}). However, one does not gain more information in doing so, whence we avoid discussing this tensor product (or any topology on the tensor product). For this reason one of the authors has been accused of ``cheating''.}
  Though we frequently write $\lcB \otimes \lcB$ also for a locally convex algebra $\lcB$ this will always denote the algebraic tensor product.
 \end{remark}

 \begin{example}
  Every Banach algebra and thus in particular every finite-dimensional algebra (over $\K \in \{\R,\C\}$) is a locally convex algebra.
  Further examples are constructed in Example \ref{ex: mero} below.
 \end{example}

 \begin{lemma}\label{lem: top:gpLA}
  Let $\Hopf$ be a Hopf algebra and $\lcB$ be a locally convex algebra.
  Then the topology of pointwise convergence turns $(\Hom_\K (\Hopf , \lcB) , \star)$ into a locally convex algebra, $(\Char{\Hopf}{\lcB})$ into a topological group and $(\InfChar{\Hopf}{\lcB}, \LB{})$ into a topological Lie algebra.
 \end{lemma}

 \begin{proof}
  The multiplication $\star$ of the character group and the Lie bracket of $\InfChar{\Hopf}{\lcB}$ will be continuous if the algebra product $\star$ on $\Hom_\K (\Hopf, \lcB)$ is continuous.
  Since this algebra is endowed with the topology of pointwise convergence, it suffices to test continuity of $(f,g) \mapsto f \star g$ by evaluating in an element $c \in \Hopf$.
  Using Sweedler notation (see \cite{MR0252485}) for the coproduct we write $\Delta (c) = \sum_{(c)} c_{1} \otimes c_{2}$, whence $f\star g (c) = \sum_{(c)} f(c_{1}) \cdot g(c_{2})$ where the dot is the multiplication in $\lcB$.
  Since point evaluations are continuous on $\Hom_\K (\Hopf,\lcB)$ and the multiplication on $\lcB$ is continuous, we deduce that $\star$ is continuous.

  Observe that inversion in the character group $(\Char{\Hopf}{\lcB},\star)$ is given by precomposition with the antipode.
  Clearly this map is continuous with respect to the topology of pointwise convergence, whence $\Char{\Hopf}{\lcB}$ is a topological group. \qed
 \end{proof}

 Note that as a by-product of Lemma \ref{lem: top:gpLA} we also obtain that $(\Hom_\K (\Hopf,\lcB),\star)$ is a locally convex algebra if $\lcB$ is a locally convex algebra.
 The last proof is a nice example of how continuity can be determined by ``testing'' it at a given point.
 In particular, the topology on the group relegates continuity questions to the topology of the target algebra $\lcB$.

 Our aim is now to turn the topological group $\Char{\Hopf}{\lcB}$ into a Lie group modelled on the infinitesimal characters.
 To do this we will need some additional assumptions on the source Hopf algebra:

 \begin{definition}
  Let $\Hopf = \bigoplus_{n \in \N_0} \Hopf_n$ be a graded and connected Hopf algebra and $\lcB$ be a locally convex algebra.
  Then we define $\cI := \Hom_\K (\bigoplus_{n \in \N} \Hopf_n , \lcB)$ and
    \begin{displaymath}
     \exp \colon \cI \rightarrow \Hom_{\K} (\Hopf, \lcB), \psi \mapsto \sum_{k=0}^\infty \frac{\psi^{\star_k}}{k!} \mbox{ where } \psi^{\star_k} := \underbrace{\psi \star \psi \star \cdots \star \psi}
    \end{displaymath}
 \end{definition}

 \begin{lemma}\label{lem: exp:ana}
  For each $\psi \in \cI$, $\exp(\psi)$ is contained in $\Hom_\K (\Hopf, \lcB)$. 
  Hence it makes sense to define $\exp \colon \Hom_{\K} (\bigoplus_{n \in \N} \Hopf_n , \lcB) \rightarrow \Hom_\K (\Hopf, \lcB)$ and this mapping is $\K$-analytic.
  Furthermore, it restricts to a bijection $\exp \colon \InfChar{\Hopf}{\lcB} \rightarrow \Char{\Hopf}{\lcB}$.
 \end{lemma}

 \begin{proof}
   We have the following isomorphisms of locally convex spaces
   \begin{displaymath}
    \Hom_\K (\Hopf, \lcB) = \Hom_\K (\bigoplus_{n \in \N_0} \Hopf_n , \lcB) \cong \prod_{n \in \N_0} \Hom_\K (\Hopf_n , \lcB).
   \end{displaymath}
  Note that $\Hom_\K (\Hopf, \lcB)$ is \textbf{not} a graded algebra (as a grading requires it to decompose into a direct sum).
  However, the convolution satisfies
  \[\Hom_\K (\Hopf_n , \lcB) \star \Hom_\K (\Hopf_m , \lcB) \subseteq \Hom_{\K} (\Hopf_{n+m},\lcB) \quad \hbox{ for all } m,n\in\N_0.
  \]
  We call an algebra with this property a \emph{densely graded algebra} (to express that it decomposes as a product and not as a direct sum).

  Now for each $n \in \N$ the formula for $\exp$ yields only finitely many summands in $\Hom_\K (\Hopf_n , \lcB)$ (as elements in its domain contain no contribution of $\Hom_\K (\Hopf_0 , \lcB)$).
  We deduce that $\exp$ makes sense as a mapping to $\prod_{n \in \N_0} \Hom_\K (\Hopf_n,\lcB)$ whence it makes sense as a mapping into $\Hom_\K (\Hopf, \lcB)$.
  To see that $\exp$ is $\K$-analytic one uses functional calculus for densely graded algebras (see \cite[Lemma B.6]{BDS15} for details).

  That $\exp$ restricts to a bijection $\InfChar{\Hopf}{\lcB} \rightarrow \Char{\Hopf}{\lcB}$ is well known in the literature. We refer to \cite[Lemma B.11]{BDS15} for a proof. \qed
 \end{proof}

 \begin{remark}
  Densely graded algebras as discussed in the proof of Lemma \ref{lem: exp:ana} are an important tool in the investigation of character groups of Hopf algebras.
  Indeed they can be used to link the investigation of character groups to the Lie theory for unit groups of continuous inverse algebras (see \cite{MR1948922,MR2997582}).
  For more details we refer to \cite{BDS15}.
 \end{remark}

 Using $\exp$ as a (global) parametrisation for $\Char{\Hopf}{\lcB}$ we obtain a manifold structure on the character group and it turns out that this renders the group a Lie group.

 \begin{theorem}[{\cite[Theorem 2.7]{BDS15}}]\label{thm: main}
 Let $\Hopf$ be a graded and connected Hopf algebra and $\lcB$ be a locally convex algebra.
 Then the manifold structure induced by the global parametrisation $\exp \colon \InfChar{\Hopf}{\lcB} \rightarrow \Char{\Hopf}{\lcB}$ turns $(\Char{\Hopf}{\lcB},\star)$ into a $\K$-analytic Lie group.

 The Lie algebra associated to this Lie group is $(\InfChar{\Hopf}{\lcB}, \LB{})$.
 Moreover, the Lie group exponential map is given by the real analytic diffeomorphism $\exp$ from Lemma \ref{lem: exp:ana}.
 \end{theorem}

 \begin{remark}
   The assumption that the Hopf algebra $\Hopf$ is graded and connected is crucial for the proof of Theorem \ref{thm: main}.
 If we drop this assumption then character groups need not be infinite-dimensional Lie groups (see \cite[Example 4.11]{BDS15} for an explicit counterexample).
 However, the construction does not depend on the choice of the grading, i.e.\ two connected gradings on a Hopf algebra will yield the same Lie group structure on the character group.
 See for example \cite[Proposition 1.30]{MR2371808} where different (connected) gradings on the Hopf algebra of Feynman graphs are discussed.
 \end{remark}

 In the literature, the infinitesimal characters are not the only Lie algebra associated with the group $\Char{\Hopf}{\lcB}$.
 Recall that for certain Hopf algebras there is a Lie algebra which is associated to the Hopf algebra by the Milnor--Moore theorem.

 \begin{remark}[The Lie algebra $\InfChar{\Hopf}{\K}$ of $\Char{\Hopf}{\K}$ and the Milnor--Moore theorem] \label{rem: MMthm}
  Consider a graded, connected and commutative Hopf algebra $\Hopf = \bigoplus_{n\in \N_0} \Hopf_n$ such that each $\Hopf_n$ is a finite-dimensional vector space.
  Then the Milnor--Moore theorem \cite{MR0174052} asserts that there is a Lie algebra $\mbox{Lie} (\Hopf)$ such that $\Hopf^\vee \cong \cU(\mbox{Lie} (\Hopf))$, where $\Hopf^\vee \cong \bigoplus_{n \in \N_0} \Hom_\K (\Hopf_n,\K)$ is the restricted dual.
  This Lie algebra is often associated with the Lie algebra of the character group $\Char{\Hopf}{\K}$ and indeed one can identify
    \begin{displaymath}
     \mbox{Lie} (\Hopf) \cong \InfChar{\Hopf}{\K} \cap \bigoplus_{n \in \N_0} \Hom_\K (\Hopf_n, \K) = \InfChar{\Hopf}{\K} \cap \Hopf^\vee.
    \end{displaymath}
  Note that the Lie algebra $\InfChar{\Hopf}{\K}$ will in general be strictly larger than $\mbox{Lie} (\Hopf)$ but with respect to the topology of pointwise convergence $\mbox{Lie} (\Hopf)$ is a dense subalgebra of $\InfChar{\Hopf}{\K}$.
  This fact is frequently exploited (cf.\ \cite{CK98,MR2371808,MR2346399}) as it connects the Hopf algebra $\Hopf$ with the Lie algebra $\InfChar{\Hopf}{\K}$.
 \end{remark}

 Before we continue with the general theory let us develop some examples which arise from applications in numerical analysis and mathematical physics.
 In Example \ref{ex: charactergroups}, we have already seen character groups of (graded and connected) Hopf algebras which arise naturally in applications.
 Most of these groups are constructed with the ground field $\K$ as target algebra.
 However, in the theory of renormalisation of quantum field theories one also considers characters with values in an infinite-dimensional locally convex algebra $\lcB$.
 We construct this algebra in the next example.
\newcommand{\Poly}{\mathcal P}

 \newcommand{\gz}{\text{germ}_0}
 \begin{example}\label{ex: mero}
  Consider the \emph{set of germs of meromorphic functions} in $0 \in \C$
    \begin{displaymath}
     \Mero = \{\mbox{germ}_0 f \mid f \colon  U \rightarrow \C\cup\{\infty\} \mbox{ meromorphic and } 0 \in U \subseteq \C \mbox{ open}\},
    \end{displaymath}
   where as usual $\mbox{germ}_0 f = \mbox{germ}_0 g$ if $f$ and $g$ coincide on some $0$-neighbourhood.
  Then $\Mero$ can be made a locally convex space as an inductive limit of Banach spaces such that
    \begin{enumerate}
     \item $\Mero$ is a complete locally convex algebra with respect to multiplication of germs,
     \item as a closed subspace of $\Mero$, the set $\Holo$ of germs of holomorphic functions is the inductive limit of the Banach spaces $\{(\BHol (U_n, \C),\norm{\cdot}_\infty ) \}_{n \in \N}$ of bounded holomorphic function on open, relatively compact sets $U_n$ which form a base of zero neighbourhoods. In particular, this structure turns $\Holo$ into a locally convex subalgebra.
    \end{enumerate}
  Character groups of graded connected Hopf algebras with values in $\Mero$ and $\Holo$ are studied in the context of the Connes--Kreimer theory of renormalisation (cf.\ \cite[p.83]{MR2371808}).
  Note that we can apply Theorem \ref{thm: main} to turn these character groups (cf.\ Example \ref{ex: charactergroups}) into Lie groups.
 \end{example}
 \begin{proof}[Construction of the topology on $\Mero$]
  \textbf{Step 1:} \emph{$\Holo$ as a locally convex algebra.}

  Fix $U_n = B_{\frac{1}{n}}^\C (0)$ for $n \in \N$ (open ball in $\C$ around $0$ of radius $\tfrac{1}{n}$). 
  For later use observe that $U_n$ is relatively compact, $\overline{U_{n+1}} \subseteq U_n$ for all $n$ and $\{U_n\}_{n \in \N}$ is a base of zero-neighbourhoods.
  Recall that the space $\BHol (U_n , \C)$ of bounded holomorphic functions is a Banach algebra with respect to the supremum norm.
  The space $\Holo$ of germs of holomorphic functions around $0$ can be realised as the inductive locally convex limit of the spaces $\BHol(U_n,\C)$.
  The bonding maps of the inductive system are given by
  \[
   \iota_{n,m} \colon \BHol(U_n,\C) \rightarrow \BHol(U_{m},\C), \quad f \mapsto f|_{U_{m}}
  \]
  and these maps are compact operators for $m>n$ (see \cite[Section 8]{MR1471480} or \cite[Appendix A]{DS15}).
  Hence $\Holo$ is a so called \emph{Silva space}.
  These spaces have many nice properties (cf.\  \cite[52.37]{MR1471480} and \cite{MR0287271}), in particular, they are Hausdorff and complete.
  Bilinear mappings on Silva spaces are continuous if they are continuous on each step of the limit (see e.g.~\cite[Proposition 4.5]{MR2743766}).
  Thus multiplication of germs is continuous and $\Holo$ is a locally convex algebra.
  \medskip

  \textbf{Step 2:} \emph{Polynomials $\Poly^\infty (X)$ without constant term as a Silva space.}

  Consider the space of polynomials $\Poly^\infty(X):=\C[X]=\text{span}\{X^0,X^1,X^2,\ldots\}$ in the formal variable $X$ and let us denote by $\Poly^\infty_*:=\text{span}\{X^1,X^2,\ldots\}$ the subspace of polynomials without constant term.
  This last space is a Silva space as the direct union $\Poly^\infty_*(X)=\bigcup_{n\in\N} \Poly^n_*(X)$ of finite dimensional spaces $\Poly^n_*(X):= \text{span}(X^1,\ldots, X^n)$.
  The bonding maps $\Poly^n_*(X)\to \Poly^{m}_*(X)$ are compact operators for $m\geq n$ as the corresponding spaces are finite dimensional.
  \medskip

  \textbf{Step 3:} \emph{The topology on $\Mero$.}

  Consider a meromorphic function $g \colon \Omega \rightarrow \C\cup\{ \infty\}$ defined on a $0$-neighbourhood $\Omega$.
  Since the set of poles is discrete, we find an $n\in\N$ such that $g$ has no poles in $\overline{U_n}\subseteq \Omega$, except possibly at $0$.
  Furthermore, write $g=f+p(1/z)$ where $f\in\BHol(U_n,\C)$ and $p\in\Poly^\infty_*(1/z)$ is a polynomial in the variable $1/z$ (without constant term).

  As $g$ was arbitrary,  the vector space $\Mero$ is a direct sum of the two vector subspaces $\Holo$ and $\Poly^\infty_*(1/z)$.
  Both of the summands have a natural Silva space structure, which can be used to topologise $\Mero =\Holo\oplus \Poly^\infty_*(1/z)$.
  The locally convex space
  \[
   \Mero=\lim_{\to} \left( \BHol(U_n,\C)\oplus \Poly^n_*(1/z) \right)
  \]
  thus becomes a Silva space as a sum of two Silva spaces (see \cite[Corollary 8.6.9]{MR880207}).

  This  construction can also be found in \cite{GrosseErdmann} where also spaces of meromorphic functions (rather than just germs of such functions) are given a topology using a similar construction.
  For meromorphic functions the multiplication turns out to be only separately continuous (see \cite[Theorem 5]{GrosseErdmann}) and so these spaces are no locally convex algebras in the sense of Definition \ref{defn_locally_convex_algebra}.
  \medskip

  \textbf{Step 4:} \emph{$\Mero$ as a locally convex algebra.}

  We now exploit that the space of germs $\Mero$ is a Silva space by Step 3 to prove that multiplication is (jointly) continuous.
  Computing on the steps of the limit, fix $n\in \N$ and set $E_n := \BHol(U_n,\C)\oplus\Poly^n_*(1/z)$.
  Let us now prove that the multiplication map (co-)restricts to a continuous map $\mu_n \colon E_n \times E_n \rightarrow E_{2n}$.
  As a tool we use the operator
  \[
   \Phi_n \colon \BHol(U_n,\C)\oplus \Poly^n_*(1/z) \rightarrow \BHol(U_n,\C), \quad f+p(1/z) \mapsto z^nf + z^np(1/z)
  \]
  which is easily seen to be linear continuous and bijective.
  Hence the Open Mapping Theorem for Banach spaces implies that $\Phi_n$ is a topological isomorphism.
  Using $\Phi_n$, we write $\mu_n$ as composition of continuous maps:
  \[
   \mu_n=\Phi_{2n}^{-1}\circ \iota_{n,2n} \circ \mu_{\BHol(U_n,\C)}\circ (\Phi_n\times\Phi_n)
  \]
  where $\mu_{\BHol(U_n,\C)}$ is the continuous multiplication in $\BHol(U_n,\C)$.
  This shows that $\Mero$ is a locally convex algebra. \qed
 \end{proof}

 \begin{remark}
  It should be noted that although $\Mero$ is algebraically a field and a locally convex algebra, inversion is \emph{not} continuous with respect to the topology just described, hence $\Mero$ is not a topological field. This however is no defect of this particular construction but follows from the (locally convex) Gelfand Mazur Theorem: There is no complex locally convex division algebra -- except $\C$ (see e.g.~\cite[Remark 4.15]{MR1948922} or \cite[Theorem 1]{ArensLinearTopologicalDivisionAlgebras}).
 \end{remark}

 We have already encountered the character group of the shuffle Hopf algebra (see Example \ref{ex: shuffle} and Example \ref{ex: charactergroups}).
 Recently, these groups were used as building blocks for a group of extended word series which is of interest in the discretisation of dynamical systems and in the computation of normal forms for these systems (see \cite{1502.05528v2,MR3485151}).
 In the following example we will revisit this construction.

 \begin{example}\label{ex: shuffle:char}
  Let $\cG := \Char{\text{Sh} (A)}{\C}$ be the complex valued character group of the shuffle Hopf algebra $\text{Sh} (A)$.
  Since $\text{Sh} (A)$ is a graded and connected Hopf algebra by Example \ref{ex: shuffle}, Theorem \ref{thm: main} implies that $\cG$ is a Lie group with the topology of pointwise convergence.
  The goal pursued in \cite{MR3485151} is to study a certain class of ordinary differential equation.
  This leads one to consider so called ``extended word series''.
  These series are elements in a semidirect product $\cG \rtimes_\Xi \C^d$ of the groups $\cG$ and $\C^d$ for some $d \in \N$.
  Here the group morphism $\Xi \colon \C^d \rightarrow \text{Aut} (\cG), z \mapsto \Xi_z$ is given for $\delta \in \cG$, $z := (z_1, \ldots, z_d) \in \C^d$ and $w = a_n\ldots a_1\in A^*$ by
    \begin{align*}
     \Xi_z (\delta) (w) = \exp\left(\sum_{k=1}^d z_k (\nu_{k, a_1} + \ldots +\nu_{k,a_n})\right) \cdot \delta (w)
    \end{align*}
  where the numbers $\nu_{k,a} \in \C$ are fixed for all $1 \leq k \leq d$ and $a \in A$.\footnote{These numbers depend on the structure of a certain ordinary differential equation. We refer to \cite{MR3485151} for more details.}
  Note that the map $\Xi$ takes its image indeed in $\text{Aut} (\cG)$ by virtue of \cite[4.2]{MR3485151}.
  Moreover, it is easy to see that the mapping $\Xi^\wedge \colon \C^d \times \cG \rightarrow \cG , (z ,\delta) \mapsto \Xi_z (\delta)$ is smooth, i.e.\ $\Xi^\wedge$ is a Lie group action.
  Hence we obtain a semidirect product of Lie groups $\cG \rtimes_\Xi \C^d$.
  In \cite{MR3485151} the authors then proceed to study the Lie algebra $\Lf (\cG \rtimes_\Xi \C^d) = \Lf (\cG) \rtimes_{\Lf (\Xi)} \C^d$, differential equations on $\cG \rtimes_\Xi \C^d$ and properties of the Lie group exponential.
 \end{example}

 Finally, one can show that certain closed subgroups of $\Char{\Hopf}{\lcB}$ which are associated to Hopf ideals are Lie subgroups.\footnote{Contrarily to the situation for finite-dimensional Lie groups, not every closed subgroup of an infinite-dimensional Lie group is again a Lie subgroup. See \cite[Remark IV.3.17]{MR2261066} for an example.}

 \begin{definition}[Hopf ideal and annihilator] \label{defn: Hideal}
Let $\Hopf$ be a Hopf algebra.
We say $\HIdeal\subseteq \Hopf$ is a \emph{Hopf ideal} if the subset $\HIdeal$ is
\begin{enumerate}
\item a two-sided (algebra) ideal,
\item a coideal, i.e.\ $\epsilon(\HIdeal)=0$ and $\Delta(\HIdeal) \subseteq \HIdeal\otimes \Hopf + \Hopf \otimes \HIdeal$ and
\item stable under the antipode, i.e.\ $S(\HIdeal)\subseteq \HIdeal$.
\end{enumerate}
Let $\lcB$ be a locally convex algebra, then we define the \emph{annihilator of $\HIdeal$}
 \[
  \Ann(\HIdeal, \lcB)= \{\phi \in \Hom_\K(\Hopf,\lcB) \mid \phi(\HIdeal)=0_\lcB\},
 \]
 which is a closed unital subalgebra of $\Hom_\K(\Hopf,\lcB)$.
\end{definition}

\begin{proposition}[{\cite[Theorem 3.4]{BDS15}}]\label{prop: ann:sbgp}
 Let $\Hopf$ be a graded connected Hopf algebra, $\HIdeal\subseteq \Hopf$ be a Hopf ideal and $\lcB$ be a commutative locally convex algebra.

 Then $\AnnGroup \subseteq \Char{\Hopf}{\lcB}$ is a closed Lie subgroup whose Lie algebra is $\AnnLieAlgebra \subseteq \InfChar{\Hopf}{\lcB}$.
 Moreover, $\AnnGroup$ is even an exponential BCH--group.
\end{proposition}

\begin{example}[Annihilator subgroups]
 \begin{enumerate}
  \item In \cite{MR2350437} Hopf ideals of the Hopf algebra of Feynman graphs $\Hopf_{FG}$ (cf.\ Example \ref{ex: charactergroups} b)) are studied. In the physical theory these ideals implement the so called ``Ward--Takahashi'' and ``Slavnov--Taylor'' identities.
  Then in \cite{MR2520515} character groups related to the annihilator subgroups of these Hopf ideals are studied.
  \item It is well known that the tree maps associated to numerical integration schemes which preserve certain first order integrals form a subgroup of the Butcher group, called the \emph{group of symplectic tree maps} $S^\K_{\text{TM}}$.
  This subgroup is the annihilator subgroup of a certain Hopf ideal in the Hopf algebra of rooted trees $\Hopf_{CK}^\K$ (see Example \ref{ex: RT}).
  We refer to \cite[Example 4/9]{BDS15} for more details.
  \item Recently in \cite{MR2946461} an even smaller subgroup of the group of symplectic tree maps (as discussed in 2.) has been considered.
  This group $\widehat{\cG}$ consists of all elements in the Butcher group such that the operations forming B-series and changing variables in the vector field commute (see \cite[Proposition 3.1]{MR2946461} for the detailed statement).
  Moreover, following \cite[Eq. (56)]{MR2946461} one can prove that elements in $\widehat{\cG}$ admit a more ``compact'' expansion as B-series.

  Due to loc.\ cit.\ and \cite[Remark 24]{MR2271214} this group is the annihilator of a Hopf ideal $\cI$ in $\Hopf_{CK}^\K$, whence a Lie group by Proposition \ref{prop: ann:sbgp}.
  Further, we remark that the ideal $\cI$ is of interest in its own right:
  As pointed out in \cite[Section 7]{MR2271214} the ideal $\cI$ appears as the kernel of a Hopf algebra morphism used to interpret the theory developed in loc.\ cit.\ in the context of Hopf algebras.
 \end{enumerate}
\end{example}

 \section{Notes on the topology of the character groups}

  By definition of the topology of pointwise convergence, the topology and the differentiable structure of the character group $\Char{\Hopf}{\lcB}$ is completely determined by the target algebra $\lcB$.
  Of course the algebraic structure of $\Hopf$ determines the set of characters, e.g.\ it controls whether the character group is an abelian group.
  However, we do not need a topology on the Hopf algebra to turn its character groups into Lie groups.
  In a nutshell, the main idea behind the construction can be described as:
  The Lie group structure is controlled by the combinatorial data of the Hopf algebra and the topological data of the target algebra.

  \begin{remark}
  Note that the topology of pointwise convergence is a very natural choice for a topology on the character groups of Hopf algebras.
  In this respect certain character groups with this topology have already been studied as topological groups in the literature.
  	\begin{enumerate}
  	\item   For example in the Connes--Kreimer theory of renormalisation (cf.\ \cite[Proposition 1.47]{MR2371808}) the structural properties of certain character groups as topological groups with this topology are exploited.
  	We refer to Section \ref{sect: pro-Lie} for further information on this structure.
   \item As a further example we mention \cite[3.2]{MR3350091} where the projective limit topology on the character groups coincides with the topology of pointwise convergence with respect to the \emph{discrete} topology on the ground field $\K$. Hence it does not coincide with our topology which induces on every one-dimensional subspace the (natural) metric topology of $\R$ (or $\C$).
  Since \cite{MR3350091} works with an arbitrary field $\K$ of characteristic $0$ the discrete topology on $\K$ is the right choice in that setting.
  \end{enumerate}
  \end{remark}

  The topology of pointwise convergence is rather coarse, i.e.\ compared to other natural function space topologies it has few open sets.
  In particular, open sets only control the behaviour of characters in a finite number of points at once.
  Especially in applications in numerical analysis one would like to have finer topologies which enable a better control.
  Let us illustrate this in the example of the Butcher group:

  \begin{example}
   The Butcher group $\BGp$ coincides with the group $\Char{\Hopf_{CK}^\K}{\K}$ of $\K$-valued characters of the Connes--Kreimer Hopf algebra of rooted trees (see Example \ref{ex: charactergroups} (c)).
   This Hopf algebra is graded and connected, whence Theorem \ref{thm: main} allows $\BGp$ to be turned into a Lie group with the pointwise topology.
   Recall that as an algebra $\Hopf_{CK}^\K$ is the polynomial algebra $\K [\RT]$ (where $\RT$ is the set of rooted trees).
   The elements in the Butcher group correspond to numerical integration schemes as they are linked to a certain type of (formal) series called $B$-series.
   In applications one now wants to restrict the growth of the series coefficients to achieve convergence of the series at least on a small disk.
   To this end, one commonly imposes an exponential growth bound to the elements in the Butcher group, i.e.\ the growth of a character is restricted by an exponential bound in every tree.
   The topology of pointwise convergence does not contain open sets which allow one to control infinitely many coefficients at once, whence it is too coarse for some applications.
   However, no suitable replacement for the topology on $\BGp$ to circumvent these problems is presently known (see the discussion of topologies on the Butcher group in \cite[Remark 2.5]{BS14}).
  \end{example}

   Though there seems to be no candidate for a finer topology on character groups which turns these into topological groups, the situation is better if one considers only certain subgroups.
   Again we specialise to the case of the Butcher group.

  \begin{example}[The tame Butcher group]
  Let $\cB$ be the subgroup of $\Char{\Hopf_{CK}^\K}{\K}$ of all elements $\varphi$ which satisfy
    \begin{displaymath}
   \text{there exist } C,\, K >0 \mbox{ such that } \norm{\varphi (\tau)} \leq C K^{|\tau|} \mbox{ for all } \tau \in \RT_0.
    \end{displaymath}
   Adapting a result of Hairer and Lubich \cite[Lemma 9]{HL1997}, one can show that the B-series associated to elements in $\cB$ with respect to an analytic map $f$ converge at least locally.
   We call $\cB$ the \emph{tame Butcher group}.

   One can show that the tame Butcher group is a Lie group modelled on an inductive limit of Banach spaces.
   Note that the resulting topology is strictly finer than the topology of pointwise convergence.

   Albeit the differential structure is more complicated than the one of the Butcher group, the tame Butcher group is closely related to the Butcher group.
   The key property of the tame Butcher group is that it provides a better control for the purposes of numerical analysis.
   Furthermore, consider on a Banach space $E$ the differential equation
   \begin{displaymath}
    \begin{cases}
     \frac{\text{d}}{\text{d} t} x(t) &= f(x(t))\\
     x(0) &= y_0
    \end{cases} \text{ with } f \colon E \rightarrow E \text{ analytic}.
   \end{displaymath}
   Then the map sending an element of $\cB$ to the B-series with respect to $f$ induces a Lie group (anti)morphism
    \begin{displaymath}
     B_f \colon \cB \rightarrow \text{DiffGerm}_{(y_0,0)} (E\times \C) := \left\{\text{germ}_{(y_0,0)} \phi \middle| \substack{\phi \colon E \times \K \supseteq U \rightarrow V \text{ is a diffeomorphism}\\ \text{ with }\phi (y_0,0) = (y_0,0)}\right\},
    \end{displaymath}
   where the group operation of $\text{DiffGerm}_{(y_0,0)} (E\times \C)$ is composition (cf.\ \cite[Section 3]{MR2670675}).
   Then the map $B_f$ can be seen as the Lie theoretic realisation of the mechanism which passes from a numerical integrator to the associated numerical solution of the differential equation.
   We refer to \cite{BS15} for further details.
   \end{example}

  Conceivably one can adapt the idea of the tame Butcher group to the general setting of character groups of Hopf algebras.

    \begin{problem}
    It would be interesting to see whether the construction of the tame Butcher group can be generalised to obtain a Lie theory for ``groups of exponentially bounded characters'' of graded connected Hopf algebras.
    The groups we envisage here should arise as subgroups of character groups (albeit with a finer topology/different differential structure).

    We expect these groups to be useful in several applications.
    In particular, one would hope that such a theory is applicable in the following situations:
      \begin{enumerate}
       \item In the context of Lie--Butcher theory of numerical analysis, i.e.\ for numerical integrators on manifolds and Lie groups (see e.g.\ \cite{MR3085679} and the references therein).
       \item Often in the theory of numerical analysis bounds appear naturally.
    For an example of such a situation see e.g.\ \cite{MR3320934} where bounds on a function and its Fourier coefficients are used to derive error estimates.
       \item In control theory so called output feedback equations are studied. Recently connections of problems related to these equations with character groups of Hopf algebras have been discovered (see \cite{MR3278760} and \cite{1507.06939v1}).
       In particular, one is interested in the convergence of certain formal series, which is assured by similar growth conditions as imposed in the tame Butcher group (cf.\ \cite[Section 2.2]{MR3278760}).
      \end{enumerate}
    \end{problem}

  Another candidate for a topology on the Butcher group $\BGp$ appears implicitly in Butcher's 1972 paper.
   \begin{theorem}[{\cite[Theorem 6.9]{Butcher72}}]
   If $a \in \BGp$ and $\RT_f$ is any finite subset of $\RT$ then there is a $b\in G_0$ such
   that $a\vert_{\RT_f}=b\vert_{\RT_f}$.
   \end{theorem}
   The above theorem is sometimes paraphrased as saying that the set of Runge--Kutta methods is dense in $\BGp$ \footnote{$G_0$ in \cite[Theorem 6.9]{Butcher72} is larger than the set of Runge--Kutta methods, however, the statement still holds if $G_0$ is the group of Runge--Kutta methods. See also \cite[Theorem 317A]{Butcher08}.}.

   Indeed, \cite[Theorem 6.9]{Butcher72} implies that the set of Runge--Kutta methods is dense in the Butcher group equipped with the topology of pointwise convergence.
   However, it also implies that the set of Runge--Kutta methods stays dense if one uses a much finer topology on $\BGp \subset \Hom_\K(\Hopf_{CK}^\K, \K)$, the ultrametric topology.

   \begin{definition}[Ultrametric topology]
 When $\Hopf$ is graded, the ultrametric topology on  $\Hom_\K (\Hopf, \lcB)$ is generated by the \emph{ultrametric}\footnote{A metric $d$ is an ultrametric if $d(\phi, \psi) \le \max\{ d(\phi, \chi), d(\chi, \psi)\}$.}
   \[d(\phi, \psi)=2^{-\textrm{ord}(\phi-\psi)},\]
   where $\textrm{ord}(\phi)$ is the largest $N\in \N_0\cup \{\infty\}$ such that
   \[\phi(x)=0 \quad \text{for all } x\in \bigoplus_{n=0}^{N} \Hopf_n.\]
   \end{definition}
 For $\BGp \subset \Hom_{\K}(\Hopf_{CK}^\K,\K)$, define the ultrametric topology as the subspace topology.
 In the ultrametric topology on $\BGp$, a sequence $(a_n)_{n \in \N}$ will converge to $a$ only if, for every tree $\tau \in \RT$, there is an $N$ such that $N \leq n$ implies $a_n(\tau)=a(\tau)$.
  Note that the ultrametric topology and the topology of pointwise convergence behave quite differently:
 Let $\Hopf$ be a $\K$-Hopf algebra with $\K \in \{\R,\C\}$.
 If we embed $\K$ as a linear subspace into $\Hom_\K (\Hopf, \lcB)$, the ultrametric topology induces the discrete topology on $\K$ whereas the topology of pointwise convergence induces the usual (metric) topology on $\K$.
 In particular, this shows that the ultrametric topology does not turn $\Hom_\K (\Hopf, \lcB)$ (and $\InfChar{\Hopf}{\lcB}$) into locally convex spaces over $\K$.

 The ultrametric topology can be useful in numerics; for instance, the order of a B-series method given by $a \in \BGp$ can be read off directly from the ultrametric via $d(a, e)= 2^{-p}$, where $e$ is the ``exact'' method and $p$ is the order of $a$.
 We record a further difference between both topologies in the following example:

\begin{example}
Let $a_h$ denote the B-series $a_h(\emptyset)=1$, $a_h(\onenode)=h$, $a_h(\tau)=0$ for all trees with $|\tau|\ge 2$, and consider the sequence of B-series
\[\{b_n\}_{n\in \N},\quad b_n=a_{1/n}^{\star n}.\]
$b_n$ corresponds to the numerical method obtained by taking $n$ steps with the forward Euler method with stepsize $\frac 1n$.
It is possible to show that
\[\lim_{n \rightarrow \infty} b_n(\tau)=e(\tau), \quad \text{for all $\tau\in \RT$}.\]
Therefore, in the topology of pointwise convergence, $\lim_{n \rightarrow \infty}b_n = e$, reflecting that the Euler method is consistent.

It is also possible to show that $b_n(\twonode) = \frac{n-1}{2n}$, whereas $e(\twonode)=\frac{1}{2}$.
Therefore $d(b_n, e)=2^{-1}$, reflecting that the numerical methods corresponding to $b_n$ are all first order methods.
However $d(b_n, e)=2^{-1}$ also shows that, in the ultrametric topology, $\{b_n\}_{n\in \N}$ does not converge to $e$ (or at all).
\end{example}

   \section{The exponential map and regularity of character groups}

In this section we discuss Lie theoretic properties of the Lie group of characters of a graded and connected Hopf algebra.
Namely we consider the Lie group exponential map and a property called regularity in the sense of Milnor.
The common theme of both properties is that they are related to the solution of certain differential equations on the Lie group.
\medskip

We begin with a discussion of the properties of the Lie group exponential map.
Recall that the interplay between the Lie algebra and the Lie group for infinite-dimensional Lie groups is more delicate than in the finite-dimensional case.
For example there are infinite-dimensional Lie groups whose exponential map does not define a local diffeomorphism in a neighbourhood of the unit.
See the survey in \cite{MR2261066} for more information.
However, it turns out that the situation for character groups of Hopf algebras is much better as they belong to certain well behaved classes of Lie groups which we define now.

 \begin{definition}
 Let $G$ be a Lie group with smooth exponential map $\exp_G \colon \Lf (G) \rightarrow G$.
 The Lie group $G$ is called a
 	\begin{enumerate}
 	\item \emph{(locally) exponential} Lie group if $\exp_G$ induces a (local) diffeomorphism,
 	\item \emph{Baker--Campbell--Hausdorff}-Lie group (or BCH-Lie group for short), if $G$ is a $\K$-analytic (locally) exponential Lie group and $\exp_G$ induces a local $\K$-analytic diffeomorphism at $0$.
 	\end{enumerate}
   \end{definition}

 In a BCH-Lie group the Baker--Campbell--Hausdorff-series converges on an open zero-neighbourhood and defines an analytic multiplication on this neighbourhood.
 The BCH-formula is often used in applications of Lie algebras to numerical analysis (see e.g.\ \cite[III.4 and III.5]{MR2221614}).
 However, there are even applications of the BCH-formula associated to character groups of certain Hopf algebras.

 \begin{remark}
  In \cite{MR2271214} computations of the BCH-formula in an arbitrary Hall basis using labelled rooted trees are presented.
  It turns out that the BCH-formula discussed there is related to the BCH-formula on the Lie algebra of infinitesimal characters on certain Hopf algebras of labelled trees (cf.\ \cite[Section 7]{MR2271214} and in particular loc.\ cit.\ Example 10).
  Also see \cite{MR2510918} for further information and the references contained therein.
 \end{remark}

 From a Lie theoretic point of view BCH-Lie groups have very strong structural properties.
 For example the automatic smoothness theorem \cite[Proposition 2.4]{MR1934608} implies that every continuous group homomorphism between BCH-Lie groups is automatically $\K$-analytic. In particular this entails that the structure of a BCH-Lie group as a topological group uniquely determines the Lie group structure.

 \begin{proposition}[{\cite{BDS15}}]
  Let $\Hopf$ be a graded and connected Hopf algebra and $\lcB$ be a commutative locally convex algebra. \begin{enumerate}
   \item Then the Lie group $\Char{\Hopf}{\lcB}$ is a BCH-Lie group,
   \item the exponential map $\exp_{\Char{\Hopf}{\lcB}}$ is a $\K$-analytic diffeomorphism.
  \end{enumerate}
 \end{proposition}

\begin{remark}[Applications of the character group exponential map]\mbox{}
\begin{enumerate}
 \item The Lie group exponential of the Butcher group (cf.\ Example \ref{ex: charactergroups} (c)) has been used in computations in numerical analysis.
Namely, in backward error analysis one exploits that this map is a diffeomorphism  (see e.g.\ \cite[Proposition 8]{MR2271214}).

The inverse of the exponential map is closely related to so called Lie derivatives of B-series (see \cite[IX.1]{MR2221614}).
In particular, the formula in \cite[Lemma 9.1]{MR2221614} for the Lie derivative can be identified with the recursion formula for the inverse of the exponential map (cf.\ \cite[Section 5]{BS14}).
However, albeit the term ``Lie derivative'' is used in the literature, only algebraic properties of these maps are exploited.
 \item In \cite{MR3350091} related pairs of exponential maps of character groups have been studied in the context of the universal enveloping algebra of a post-Lie algebra.
 These mappings play a role in the numerical integration on post-Lie algebras.
 However, we should mention at this point that the topology on the character groups used in loc.cit. differs from the one we used in the construction of the Lie group structure.
\end{enumerate}
\end{remark}

\begin{problem}
 We have already mentioned several times that the Hopf algebras considered in numerical analysis are connected to so called pre- and post-Lie algebras.
 Recently these structures have gathered a lot of interest, see e.g.\ the work by Munthe-Kaas and collaborators \cite{MR3350091,MR3085679}.
 It would be interesting to see whether these additional structures induce more structure which is visible in the Lie group structure of the character groups.
 \end{problem}

We now turn to regularity properties of character groups.
Roughly speaking, regularity (in the sense of Milnor) means that a certain class of (ordinary) differential equations can be solved on the Lie group.
Note that it is highly non-trivial to solve ordinary differential equations on locally convex spaces beyond the realm of Banach spaces.
For example there are linear differential equations without solution or which admit infinitely many different solutions (see \cite{milnor1982} or \cite[Section 5.5]{MR656198}).

\begin{definition}[Regularity (in the sense of Milnor)]\label{defn: regular}
Let $G$ be a Lie group modelled on a locally convex space, with identity element $\one$, and
 $r\in \N_0\cup\{\infty\}$. We use the tangent map of the left translation
 $\lambda_g\colon G\to G$, $x\mapsto gx$ by $g\in G$ to define
 $g.v:= T_{\one} \lambda_g(v) \in T_g G$ for $v\in T_{\one} (G) =: \Lf(G)$.
 Following \cite{1208.0715v3}, $G$ is called
 \emph{$C^r$-semiregular} if for each $C^r$-curve
 $\gamma\colon [0,1]\rightarrow \Lf(G)$ the initial value problem
 \begin{displaymath}
  \begin{cases}
   \eta'(t)&= \eta (t). \gamma(t)\\ \eta(0) &= \one
  \end{cases}
 \end{displaymath}
 has a (necessarily unique) $C^{r+1}$-solution $\Evol (\gamma):=\eta\colon [0,1]\rightarrow G$.
 If furthermore the map
 \begin{displaymath}
  \evol \colon C^r([0,1],\Lf(G))\rightarrow G,\quad \gamma\mapsto \Evol
  (\gamma)(1)
 \end{displaymath}
 is smooth, $G$ is called \emph{$C^r$-regular}.\footnote{Here we consider $C^r([0,1],\Lf(G))$ as a locally convex vector space with the pointwise operations and the topology of uniform convergence of the function and its derivatives on compact sets.} If $G$ is $C^r$-regular and $r\leq s$, then $G$ is also
 $C^s$-regular. A $C^\infty$-regular Lie group $G$ is called \emph{regular}
 \emph{(in the sense of Milnor}) -- a property first defined in \cite{MR830252}.
 Every finite-dimensional Lie group is $C^0$-regular (cf.\ \cite{MR2261066}).
\end{definition}

Several important results in Lie theory are only available for regular Lie groups  (see
 \cite{MR830252,MR2261066,1208.0715v3}, cf.\ also \cite{MR1471480} and the references therein).
Up to this point all known Lie groups modelled on sufficiently complete spaces (i.e.\ on Mackey complete spaces, see \cite[Chapter I.2]{MR1471480}) are regular.

\begin{example}
 Consider again the Butcher group $\BGp$, discussed in Example \ref{ex: charactergroups} (c).
 The differential equation for regularity of this Lie group takes the form of a countable system of differential equations:
 For a continuous curve $\mathbf{a} \colon [0,1] \rightarrow \Lf (\BGp)$ we seek a differentiable curve $\gamma \colon [0,1] \rightarrow \BGp$ such that
  \begin{align*} \begin{cases}
                  \gamma' (t) (\tau) &= \mathbf{a}(t)(\tau) + \displaystyle\sum_{\substack{\theta \in \RT_0 \\ |\theta| < |\tau|}} A_{\theta,\tau} (t,\mathbf{a})\gamma (t)(\theta), \\
		   \gamma (0)(\tau) &= 0
                 \end{cases} \quad  \forall \tau \in \RT
\end{align*}
 where $ A_{\theta,\tau} (t,\mathbf{a})$ is a polynomial in $\mathbf{a}(t)(\sigma), \sigma \in \RT_0$ with $|\sigma| < |\tau|$.
 These differential equations form an infinite lower diagonal system of differential equations. As shown in \cite[Section 4]{BS14} this system can be solved inductively, i.e.\ via a projective limit argument.\footnote{This is possible since the differentiable structure of the Butcher group turns it into a projective limit of finite-dimensional Lie groups.
 We will return to this phenomenon in Section \ref{sect: pro-Lie}.}
 Hence $\BGp$ is a $C^0$-regular Lie group.
\end{example}

 \begin{proposition}[{\cite{BDS15}}]\label{prop: char:reg}
  Let $\Hopf$ be a graded and connected Hopf algebra and $\lcB$ be a commutative and complete locally convex algebra.
  Then $\Char{\Hopf}{\lcB}$ is a $C^0$-regular Lie group.
  \end{proposition}

  Note that we had to require $\lcB$ to be a complete algebra in Proposition \ref{prop: char:reg}.
  One can weaken this requirement, as it turns out that $\Char{\Hopf}{\lcB}$ is still a regular Lie group for so called ``Mackey complete'' locally convex algebras.
  We refer to \cite{BDS15} for further information.

\begin{remark}
The Lie theory of character groups is closely connected to the Lie theory for unit groups of continuous inverse algebras (see \cite{MR2997582,MR1948922}).
In fact, each character group of a graded and connected Lie algebra can be identified with a closed Lie subgroup of the unit group of a suitable continuous inverse algebra. The details of this construction are recorded in \cite[Proof of Theorem 2.10]{BDS15}, where we have exploited this link to derive the regularity of the character group from the regularity of the ambient unit group.
\end{remark}

The differential equations of regularity also occur in natural questions connected to character groups of Hopf algebras.
Let us illustrate this with two examples.

Our first example is from applications in numerical analysis.
In \cite[p.8]{MR2946461} these differential equations appear naturally in the investigation of higher order averaging.

The second example appears in the theory of renormalisation of quantum field theories:

\begin{example}[Birkhoff decomposition and time ordered exponentials]\label{ex: toe}
 Consider the Hopf algebra of Feynman graphs $\Hopf_{FG}$ with its group $\Char{\Hopf_{FG}}{\C}$ of $\C$ valued characters.
 A crucial step in the Connes--Kreimer theory of renormalisation -- the so called BPHZ procedure -- can be formulated as a Birkhoff factorisation in the Lie group $\Char{\Hopf_{FG}}{\C}$.
 To this end, one wants to decompose a smooth loop, i.e.\ a smooth map $\gamma \colon C \rightarrow \Char{\Hopf_{FG}}{\C}$ defined on a circle $C\subseteq \C$ as 
 $\gamma (z) = \gamma_-^{-1}(z) \gamma_+(z)$. Here $\gamma_- ,\gamma_+$ are boundary values of certain holomorphic functions (see \cite[Definition 1.37]{MR2371808}).
 Then the negative part $\gamma_-$ of the Birkhoff decomposition yields the counterterms one seeks to compute in the renormalisation procedure (as explained in \cite[1.6.4]{MR2371808}, cf.\ explicitly \cite[Theorem 1.40]{MR2371808}).

 To prove some desirable properties of the Birkhoff decomposition one defines a \emph{time-ordered exponential}, i.e.\ for a smooth curve $\alpha \colon [a,b] \rightarrow \InfChar{\Hopf_{FG}}{\C}$ define
  \begin{displaymath}
   \mathrm{T} e^{\int_a^b \alpha (t)\D t} := \one_{\Char{\Hopf_{FG}}{\C}} + \sum_{n=1}^\infty \int_{a \leq s_1 \leq \cdots \leq s_n \leq b} \alpha (s_1)\cdots \alpha (s_n) \D s_1 \cdots \D s_n.
  \end{displaymath}
 Then it turns out that the time ordered exponentials solve the differential equation associated to regularity in $\Char{\Hopf_{FG}}{\C}$ for the curve $\alpha$ (see \cite[Proposition 1.51 (3)]{MR2371808}). 
 Time-ordered exponentials are important since they determine the Birkhoff decomposition in $\Char{\Hopf_{FG}}{\C}$. 
 Explicitly, for a loop $\gamma_\mu$ (on an infinitesimal punctured disk in $\C$) one has as negative part of the Birkhoff decomposition 
 \begin{displaymath}
  \gamma_- (z) =  \mathrm{T} e^{-\frac{1}{z}\int_0^\infty \theta_{-t} (\beta)\D t}
 \end{displaymath}
 where $\beta$ is the so called $\beta$-function of the theory (cf.\ \cite[Theorem 1.58]{MR2371808}) and $\theta$ a certain one-parameter family generated by the grading operator. 
 In particular, the time ordered exponentials determine the counterterms of perturbative renormalisation and one concludes that these depend only on the $\beta$-function of the theory.
 We are deliberately hiding the technical details here and refer instead to \cite{MR2371808}.
\end{example}

\section{Character groups as pro-Lie groups} \label{sect: pro-Lie}

A crucial requirement to turn character groups of Hopf algebras into Lie groups has been that the Hopf algebra is needed to be graded and connected.
As we have already mentioned, character groups of Hopf algebras which do not satisfy these conditions will in general not be infinite-dimensional Lie groups (with the topology of pointwise convergence).
Note however that we have already seen in Lemma \ref{lem: top:gpLA} that character groups of arbitrary Hopf algebras (with values in a locally convex algebra) are topological groups.
Similarly the Lie algebra of infinitesimal characters is a topological Lie algebra regardless of a grading on $\Hopf$.

In the present section we investigate the structure of the topological group $\Char{\Hopf}{\lcB}$ and its relation to the topological Lie algebra $\InfChar{\Hopf}{\lcB}$ for Hopf algebras which are not necessarily graded.
It turns out that for certain target algebras these topological groups admit a strong structure theory which is reminiscent of finite-dimensional Lie theory.
To phrase our results let us recall the notion of a pro-Lie group (see the extensive monograph \cite{MR2337107} or the recent survey \cite{axioms4030294}).

 \begin{definition}[pro-Lie group]													\label{defn: pro_lie_group}
   A topological group $G$ is called a \emph{pro-Lie group} if one of the following equivalent conditions holds:
   \begin{enumerate}
    \item  $G$ is isomorphic (as a topological group) to a closed subgroup of a product of finite-dimensional (real) Lie groups.
    \item  $G$ is the projective limit of a projective system of finite-dimensional (real) Lie groups (taken in the category of topological groups).
   \end{enumerate}
  \end{definition}

  The class of pro-Lie groups contains all compact groups (see e.g.~\cite[Corollary 2.29]{MR3114697}) and all connected locally compact groups (Yamabe's Theorem, see \cite{MR0054613}).
  However, this does not imply that all pro-Lie groups are locally compact and the pro-Lie groups in the present paper will almost never be locally compact.

  The structure theory of pro-Lie groups mirrors to a surprising degree the structure theory of finite-dimensional Lie groups (details are recorded in \cite{MR2337107}).
  Most importantly for us, every pro-Lie group is connected to a Lie algebra.
  We recall its construction now.
  Note that in absence of a differential structure a Lie algebra to the group can not be constructed as a tangent space.

  \begin{definition}[The pro-Lie algebra of a pro-Lie group]
   Let $G$ be a pro-Lie group and $\cL(G)$ the space of all continuous $G$-valued one-parameter subgroups, endowed with the compact-open topology.
   As a projective limit of finite-dimensional Lie algebras it is naturally a locally convex topological Lie algebra over $\R$ (see \cite[Definition 2.11]{MR2337107}).
  \end{definition}

  The character group of an arbitrary Hopf algebra (with values in a finite-dimensional algebra) turns out to be a pro-Lie group.
  In these cases the pro-Lie algebra can also be identified.

   \begin{theorem}[Character groups as pro-Lie groups {\cite[Theorem 5.6]{BDS15}}]							\label{thm: character_group_is_a_pro_lie_group}
   Let $\Hopf$ be a Hopf algebra and $\lcB$ be a commutative finite-dimensional $\K$-algebra (e.g.~$\lcB:=\K$).
   Then the group of $\lcB$-valued characters $\Char{\Hopf}{\lcB}$ endowed with the topology of pointwise convergence is pro-Lie group.

   Its pro-Lie algebra is isomorphic to the locally convex Lie algebra $\InfChar{\Hopf}{\lcB}$ of infinitesimal characters via the canonical isomorphism
   \[
    \InfChar{\Hopf}{\lcB} \rightarrow \cL(\Char{\Hopf}{\lcB})  , \quad \phi \mapsto (t\mapsto \exp(t\phi)).
   \]
  \end{theorem}

  \begin{proof}[Sketch of ideas]
   Due to the fundamental theorem of coalgebras (see \cite[Theorem 4.12]{MR2035107}), one can write $\Hopf$ as a directed union of finite-dimensional coalgebras $\{C_i\}_{i \in I}$.
   On each of the spaces $\Hom_{\K} (C_i,\lcB)$ the convolution induces a locally convex algebra structure such that the unit groups satisfy
   \begin{displaymath}
    \Hom_\K (\Hopf , \lcB)^\times = \Hom_\K (\lim_{\rightarrow} C_i , \lcB)^\times \cong \lim_{\leftarrow} \Hom_{\K} (C_i,\lcB)^\times
   \end{displaymath}
  where the limit on the right hand side is taken in the category of topological groups.
  In particular, the groups $ \Hom_{\K} (C_i,\lcB)^\times$ are finite-dimensional Lie groups, whence $ \Hom_\K (\Hopf , \lcB)^\times$ is a pro-Lie group and $\Char{\Hopf}{\lcB}$ inherits this structure.
  The remaining assertions follow from direct computations involving the exponential map.\qed
  \end{proof}

  \begin{remark}
  \begin{enumerate}
   \item If we consider character groups of graded connected Hopf algebras with values in finite-dimensional algebras, we obtain two structures on the character group: The locally convex Lie group structure and the pro-Lie group structure. Fortunately the results in \cite{MR2475971} affirms that these two structures are compatible with each other. This means that the only additional information provided by the pro-Lie structure in this case is that the infinite-dimensional Lie group already constructed is a projective limit of finite-dimensional Lie groups.
    \item One can generalise the preceding result beyond the realm of finite-dimensional target algebras to so called ``weakly complete algebras''.
  See \cite[Section 5]{BDS15} for the definition and detailed statements.
  \end{enumerate}
  \end{remark}

  In applications of character groups of Hopf algebras the pro-Lie property has frequently been of crucial importance.
  Most importantly, it allows one to work with the projective limit structures of the Lie algebra of infinitesimal characters and the character group.

  \begin{example}
   \begin{enumerate}
    \item For the Hopf algebra of Feynman diagrams $\Hopf_{FG}$ one considers flat $\InfChar{\Hopf_{FG}}{\C}$-valued connections.
    To prove some desirable properties for maps depending on these connections one has to employ projective limit arguments, e.g.\ see \cite[Proposition 1.52]{MR2371808}.
    Furthermore, one can use the projective limit property to construct geometric data for elements in certain interesting algebras of infinitesimal characters.
    Here we mean the monodromy representation and their limit constructed in \cite[Lemma 1.54]{MR2371808} for elements in $\InfChar{\Hopf_{FG}}{\C}$.

    Note that in  \cite{MR2371808} the projective limit structure was deduced from the fact that every affine group scheme is a projective limit of linear algebraic groups.
    This requires the source Hopf algebra to be commutative, whereas the pro-Lie structure established in Theorem \ref{thm: character_group_is_a_pro_lie_group} does not depend on the commutativity of $\Hopf_{FG}$.
   \item In numerical analysis, properties of series are often studied by truncating to a finite number of terms, see e.g.\ the treatment of modified equations in \cite{HL1997} or \cite[Section 5.2, 5.4]{1502.05528v2}.
    This corresponds in the pro-Lie picture to passing from the projective limit to one of the (finite-dimensional) steps.
  \end{enumerate}
  \end{example}

The above list shows that it is quite useful to have the structure of a pro-Lie group to analyse character groups of Hopf algebras.
Conversely, one can ask which pro-Lie groups can be obtained as character groups.

\begin{problem}
 Characterize all pro-Lie groups that can be obtained as $\R$-valued character groups of (in general non-graded) Hopf algebras. \\
 At least all compact abelian groups appear as character groups of certain Hopf algebras (this follows from Pontryagin duality). On the other hand, one can show that an uncountable discrete group is never a character group of a Hopf algebra.
\end{problem}

\paragraph{Acknowledgement}
 The research on this paper was partially supported by the project \emph{Topology in Norway} (NRC project 213458) and \emph{Structure Preserving Integrators, Discrete Integrable Systems and Algebraic Combinatorics} (NRC project 231632).
 We thank J.M.~Sanz-Serna for pointing out references to results from numerical analysis which the authors were unaware of.
 Finally, we thank the anonymous referees for many useful comments which helped to improve the manuscript.
 \bibliographystyle{abbrv}
 \bibliography{BDS}
\end{document}